\documentclass{amsart}
\usepackage{latexsym}
\usepackage{amssymb}
\usepackage{amsfonts}
\usepackage{amsmath}
\usepackage{amsthm}
\usepackage{graphics}
\usepackage[pdftex]{graphicx}

\newtheorem{lemma}{Lemma}
\newtheorem{theorem}[lemma]{Theorem}
\newtheorem{proposition}[lemma]{Proposition}
\newtheorem{corollary}[lemma]{Corollary}

\def\b0{{\bf 0}}

\def\b{\beta}

\def\Cay{{\rm Cay}}
\def\Aut{{\rm Aut}}

\def\soc{{\rm Soc}}

\def\PSL{{\rm PSL}}

\def\PSU{{\rm PSU}}

\def\PSp{{\rm PSp}}

\def\Sz{{\rm Sz}}

\begin{document}
\openup 1.8\jot
\title[On vertex-uniprimitive non-Cayley graphs of order pq]
{On vertex-uniprimitive non-Cayley graphs of order pq}
\author[M.~A.~Iranmanesh]{Mohammad~A.~Iranmanesh}
\address{M.~A.~Iranmanesh, Department of Mathematics \newline Yazd University\\ Yazd, 89195-741 , Iran}
\email{iranmanesh@yazd.ac.ir}
\subjclass[]{20B20, 05C25}
\keywords{Uniprimitive group, Non-Cayley graph, Socle}
\begin{abstract}
Let $p$ and $q$ be distinct odd primes. Let $\Gamma=(V(\Gamma), E(\Gamma))$ be a
non-Cayley vertex-transitive graph of order $pq.$ Let $G\leq \Aut(\Gamma)$ acts
primitively on the vertex set $V(\Gamma)$. In this paper, we show that $G$ is uniprimitive
which is primitive but not 2-transitive and we obtain some information about $p, q$ and
the minimality of the Socle $T=\soc(G).$\\
\end{abstract}
\maketitle
\section{Introduction}\label{intro}
Let $\Gamma=(V(\Gamma), E(\Gamma))$ or simply $\Gamma$ be a finite, simple and
undirected graph. We denote by $\Aut(\Gamma)$, the automorphism group of $\Gamma$ and
we say that $\Gamma$ is vertex-transitive if $\Aut(\Gamma)$ acts transitively on $V(\Gamma)$.
The cardinality of $V(\Gamma)$ is called the order of the graph. For a group $G$ and a subset
$S$ of $G$ such that $1_G\notin S$ and $S^{-1}=S$, the Cayley graph $\Gamma=\Cay(G, S)$ is
defined to have vertex set $G$ and for $g, h\in G, \{g, h\}$ is an edge if and only if
$gs=h$ foe some $s\in S$. Every Cayley graph is a vertex-transitive graph and conversely, see~\cite{biggs},
a vertex-transitive graph is isomorphic to a Cayley graph for some group if and only if its
automorphism group has a regular subgroup on vertices.

By $\mathcal{NC}$ we denote the set of positive integers $n$ for which there exists a vertex-transitive
non-Cayley graph on $n$ vertices which is not a Cayley graph. In~\cite{mar1}, Maru$\check{s}$i$\check{c}$
proposed the determination of $\mathcal{NC}$ and in~\cite{mar2}, he proved that
$p, p^2, p^3\notin \mathcal{NC}$ for prime $p$. Mckay and Praeger settled the status of all
non-square-free numbers $n$ are in $\mathcal{NC}$, with the exceptions $n=12, n=p^2, n=p^3$, $p$ prime and
they finished the characterization of non-Cayley numbers which are the product of two distinct
primes (see~\cite{mp1, mp2}). The following theorem characterize all numbers $n=pq\in \mathcal{NC}$ where
$p, q$ are distinct primes. \\
\begin{theorem}\label{thm:1}~\cite[Theorem 1]{mp2}
Let $p<q$ be primes. Then $pq\in \mathcal{NC}$ if and only if one of the following holds.
\begin{itemize}
\item [$(1)$] $p^2$ divides $q-1$;
\item [$(2)$] $q=2p-1$ or $q=\frac{p^2+1}{2}$;
\item [$(3)$] $q=2^t+1$ and either $p$ divides $2^t-1$ or $p=2^{t-1}-1$;
\item [$(4)$] $q=2^t-1$ and $p=2^{t-1}+1$;
\item [$(5)$] $pq=7.11$.
\end{itemize}
\end{theorem}

For characterization of non-Cayley number which are the product of three distinct primes
we refer to~\cite{s, hip, ip}.

Suppose that a group $G$ acting on a set $V$ transitively. We say that a partition $\Sigma$ of $V$ is
$G$-invariant if $G$ permutes the blocks of $\Sigma$. A transitive action of $G$ on $V$ is said to
be primitive if all $G$-invariant partitions of $V$ are trivial. A primitive permutation group which
is not 2-transitive is said to be uniprimitive. Also a transitive permutation group is said to be
minimal transitive if all of its proper subgroups are intransitive. Suppose that $\Gamma$ is a graph of
order $pq$. In this paper we prove that if $G\leq \Aut(\Gamma)$ acts primitively on the vertex set of
$\Gamma$, then either $\Gamma$ is a Cayley graph or $G$ is uniprimitive and when $pq\notin \mathcal{NC}$
then $T=\soc(G)$ is not minimal transitive.
\section{Primitive permutation groups of degree $pq$}
First, we investigate primitive permutation groups of order $pq$ which are 2-transitive.

\begin{proposition}\label{pro:1}
Let $p$ and $q$ be distinct odd primes such that $p<q$ and suppose that $G$, a subgroup of
$S_{pq}$, is 2-transitive of degree $pq$ with Socle $T$. Then $T$ is a non-abelian simple group
and is one of the group listed below.
\begin{enumerate}
\item [$(i)$] $T=A_{pq}$;
\item [$(ii)$] $T=\PSL(n, s)$ and $pq=\frac{s^n-1}{s-1}$;
\item [$(iii)$] $T=\PSU(3, 2^a)$ and $p=2^a+1, q=2^{2a}-2^a+1$;
\item [$(iv)$] $T=\Sz(8)$ and $p=5, q=13$;
\item [$(v)$] $T=A_7$ and $p=3, q=5$.
\end{enumerate}
Also in case $(i)$ and $(ii)$ for some but not all $p$ and $q$, we have $pq\in \mathcal{NC}$. In
case $(iii)$, $pq\notin \mathcal{NC}$ except where $p=5, q=13$ and in case $(v)$, $pq\notin \mathcal{NC}$
and in case $(iv)$, $pq\in \mathcal{NC}$.
\end{proposition}
\begin{proof} By the ``O'Nan-Scott Theorem''~\cite{p}, $G$ is almost simple, that is, $G$ has a
unique minimal normal subgroup $T$ which is a non-abelian simple group. All possibilities for the
socle of almost 2-transitive groups are given for example in~\cite{c}. Note that in case $(iv)$ we
have the Suzuki groups $\Sz(q)$ with $s=2^{2a+1}=\frac{r^2}{2}\geq 8$. In this case
$pq=s^2+1= (s+r+1)(s+r-1)$. Thus 5 divides $pq$ and hence $p=5$ and $q=1$. From Theorem~\ref{thm:1} we
conclude that in case $(iv)$, $pq\in \mathcal{NC}$ and in case $(iii)$, $pq\in \mathcal{NC}$ if and only if
$a=2, p=5$ and $q=13$.
\end{proof}

\indent Now we deal with the primitive permutation groups of degree $pq$ which are not 2-transitive,
that is, the uniprimitive permutation groups of degree $pq$.
\begin{proposition}\label{pro:2}
Let $p$ and $q$ be distinct odd primes such that $p<q$. Suppose that $G$, is a uniprimitive
permutation subgroup of $S_{pq}$, of degree $pq$ with Socle $T$. Then $T$ is a non-abelian simple group
and Table~\ref{tab:1} contains all the possibilities for $T, p$ and $q$.
\end{proposition}
\begin{proof}
A similar argument to the proof of Proposition~\ref{pro:1}, we conclude that $T$ is
a non-abelian simple group. All possibilities for the socle of uniprimitive permutation groups
of degree $pq$ are given in~\cite{msz} based on~\cite{ls} which are listed in Table~\ref{tab:1}.
\end{proof}
\begin{table}[ht]
\begin{center}
\begin{tabular}{|c|c|c||c|c|c|}
\hline
T & q & p & T & q & p \\
\hline
$A_q$ & $q\geq 5$ & $\frac{q-1}{2}$ & $\PSL(2, 19)$ & 19 & 3 \\
$A_{q+1}$ & $q\geq 5$ & $\frac{q+1}{2}$ & $\PSL(2, 29)$ & 29 & 7 \\
$A_7$ & 7 & 5 & $\PSL(2, 59)$ & 59 & 29 \\
$\PSL(4, 2)$ & 7 & 5 & $\PSL(2, 61)$ & 61 & 31 \\
$\PSL(5, 2)$ & 31 & 5 & $\PSL(2, 23)$ & 23 & 11 \\
$\PSp(4, 2^a)$ & $2^{2a}+1$ & $2^a+1$ & $\PSL(2, 11)$ & 11 & 5 \\
$\Omega^{\pm}(2n, 2)$ & $2^n\mp 1>7$ & $2^{n-1}\pm 1$ & $\PSL(2, 13)$ & 13 & 7 \\
$\PSL(2, q)$ & $q\geq 13$ & $\frac{q+1}{2}$ &  $M_{23}$ & 23 & 11 \\
$\PSL(2, q)$ & $q\geq 7$ & $\frac{q-1}{2}$ & $M_{22}$ & 11 & 7 \\
$\PSL(2, p^2)$ & $q=\frac{p^2+1}{2}$ & p & $M_{11}$ & 11 & 5 \\
\hline
\end{tabular}
\end{center}
\caption{Socle of uniprimitive groups of degree $pq$.}\label{tab:1}
\end{table}

\section{Uniprimitive graphs of order $pq$}
In this section we prove that there is no any non-Cayley graph of order $pq$ which admits a 2-transitive
subgroup of automorphisms.

\begin{theorem}\label{thm:2}
Let $p$ and $q$ be distinct odd primes and $\Gamma=(V(\Gamma), E(\Gamma))$ be a vertex-transitive
graph of order $pq$. Suppose that there exists a subgroup $G$ of $\Aut(\Gamma)$ which acts 2-transitively
on $V(\Gamma)$. Then $\Gamma$ is a Cayley graph.
\end{theorem}
\begin{proof}
Let $\Gamma$ be a vertex-transitive graph of order $pq$ and suppose that $G\leq \Aut(\Gamma)$ is
2-transitive on $V(\Gamma)$. If $\Gamma$ be an empty graph of order $pq$, then obviously $\Gamma$
is a Cayley graph. So we may assume that there exist $x, y\in V(\Gamma)$, such that $\{x, y\}\in E(\Gamma)$.
Since $G$ is doubly transitive on $V(\Gamma)$, there exists $g\in G$, such that $x^g=x'$ and $y^g=y'$.
Thus $\{x, y\}^g=\{x', y'\}\in E(\Gamma)$. This shows that $\Gamma$ is a complete graph and the proof
is complete.
\end{proof}

\begin{corollary}\label{cor:1}
Let $\Gamma$ be a non-Cayley vertex-transitive graph of order $pq$. If $G\leq \Aut(\Gamma)$ acts
primitively on $V(\Gamma)$, then $G$ is uniprimitive and $T=\soc(G)$ is one of the group listed
in Table~\ref{tab:2}. Also if $pq\notin \mathcal{NC}$, then $T$ is not minimal transitive.
\end{corollary}
\begin{proof}
The first part is direct consequence of Theorem~\ref{thm:2}. If $pq\notin \mathcal{NC}$, then
$T$ is one of the groups in Table~\ref{tab:2}. For each line we see that $T$ has a transitive subgroup
isomorphic to a Frobenius group of order $pq$. This can be shown by using~\cite{ccnpw}.
\end{proof}

\begin{table}[ht]
\begin{center}
\begin{tabular}{|c|c|c|}
\hline
T & q & p \\
\hline
$A_q$ & $q\geq 5$ & $\frac{q-1}{2}$ \\
$\PSL(5, 2)$ & 31 & 5 \\
$\PSL(2, q)$ & $q\geq 7$ & $\frac{q-1}{2}$\\
$\PSL(2, 29)$ & 29 & 7 \\
$\PSL(2, 59)$ & 59 & 29 \\
$\PSL(2, 23)$ & 23 & 11 \\
$\PSL(2, 11)$ & 11 & 5 \\
$M_{23}$ & 23 & 11 \\
$M_{11}$ & 11 & 5 \\
\hline
\end{tabular}
\end{center}
\caption{Socle of uniprimitive groups of degree $pq$ where $pq\notin \mathcal{NC}$.}\label{tab:2}
\end{table}
\vskip 3mm
\noindent{\bf Acknowledgment.} The author wishes to thank Yazd
University Research Council for financial support.

\end{document}